\documentclass[12pt]{amsart}
\usepackage{amsmath}
\usepackage{amsthm}
\usepackage{amssymb}
\usepackage{amsfonts,mathrsfs}
\usepackage{stmaryrd}
\usepackage{amsxtra}
\usepackage{epsfig}
\usepackage{verbatim}

\theoremstyle{plain}
\newtheorem{theorem}[equation]{Theorem}
\newtheorem{proposition}[equation]{Proposition}

\newtheorem{corollary}[equation]{Corollary}
\newtheorem{definition}[equation]{Definition}
\theoremstyle{remark}
\newtheorem{remark}{Remark}
\numberwithin{equation}{section}

\include{header}

\newcommand{\wB}{\mathbf{B}_{\lambda}}

\newcommand{\Bmu}{\mathbf{B}_{\mu}}
\newcommand{\BOmmu}{\mathbf{B}_{\Omega,\mu}}
\newcommand{\Bom}{\mathbf{B}_{\omega}}
\newcommand{\uplane}{\mathbb{R}^2_+}
\newcommand{\Aptt}{A_p^+(\mathbb{R}^2_+)}

\newcommand{\BOm}{\mathbf{B}_{\Omega}}
\long\def\symbolfootnote[#1]#2{\begingroup%
\def\thefootnote{\fnsymbol{footnote}}\footnote[#1]{#2}\endgroup}

\begin{document}
\author{Yunus E. Zeytuncu}
\title[$L^p$ Regularity of Weighted Bergman Projections]{$L^p$ Regularity of Weighted Bergman Projections}
\address{Department of Mathematics, Texas A\&M University, College Station, TX 77843}
\email{zeytuncu@math.tamu.edu}

\begin{abstract} 
We investigate $L^p$ regularity of weighted Bergman projections on the unit disc and $L^p$ regularity of ordinary Bergman projections in higher dimensions.
\end{abstract}

\maketitle

\section{Introduction}

\subsection{Setup and Problems}
Let $\Omega$ be a domain in $\mathbb{C}^{n}$ and $\mu(z)$ be a non-negative measurable function on $\Omega$.
Let $L^2(\Omega,\mu)$ denote the space of square-integrable functions on $\Omega$ with respect to the measure $\mu(z)dA(z)$ where $dA(z)$ is the ordinary Lebesgue measure. 
We call $\mu(z)$ a \textit{weight} on $\Omega$ and $L^2(\Omega, \mu)$ \textit{the weighted $L^2$} space of $\Omega$.
$L^2(\Omega, \mu)$ is a Hilbert space with the inner product:
\begin{equation*}
\left<f,g\right>_{\mu}=\int_{\Omega}f(z)\overline{g(z)}\mu(z)dA(z),
\end{equation*}
and the norm:
\begin{equation*}
||f||_{2,\mu}^2=\int_{\Omega}\left|f(z)\right|^2\mu(z)dA(z).
\end{equation*}

\noindent Let $L^2_a(\Omega, \mu)$ denote the subspace of holomorphic functions in $L^2(\Omega, \mu)$. 
This subspace may be trivial or finite dimensional depending on the weight $\mu$. 
In such a case the main problem of this paper becomes trivial. It will be clear from the context that 
$L^2_a(\Omega, \mu)$ will be always infinite dimensional for all weights considered in this paper. 

\begin{definition} 
A weight $\mu$ is said to be an \textit{admissible weight on} $\Omega$ if   
for any compact subset $K$ of $\Omega$, there exists $C_K>0$ such that
\begin{equation*}
 \sup_{z\in K}|f(z)|\leq C_K||f||_{2,\mu}
\end{equation*}
for all f $\in L^2_a(\Omega, \mu)$.
\end{definition}
For instance if $\mu$ is continuous and never vanishes inside $\Omega$ (it can still vanish on the boundary) 
then it satisfies the inequality above and therefore it is admissible. It is easy to see that if $\mu$ is admissible 
then $L^2_a(\Omega, \mu)$ is a closed subspace of $L^2(\Omega,\mu)$ and all point evaluation maps are continuous. 
See \cite{Pasternak90} for this definition and some sufficient conditions. In this note, all weights are admissible.

When $L^2_a(\Omega, \mu)$ is a closed subspace of $L^2(\Omega, \mu)$ there exists the orthogonal projection operator that we call \textit{the weighted Bergman projection}:
\begin{equation*}\label{projection}
\mathbf{B}_{\Omega, \mu}: L^2(\Omega, \mu) \to L^2_a(\Omega, \mu).
\end{equation*}
This projection is an integral operator with the kernel, called \textit{the weighted Bergman kernel}, denoted by $B_{\Omega, \mu}(z,w)$:
\begin{equation*}\label{integraloperator}
\mathbf{B}_{\Omega, \mu}f(z)=\int_{\Omega} B_{\Omega, \mu}(z,w)f(w)\mu(w)dA(w).
\end{equation*}

When $\mu(z)\equiv 1$, we call the weighted projection \textit{the ordinary Bergman projection of} $\Omega$. We denote the space of weighted $p-$integrable functions by $L^p(\Omega, \mu)$ for $p\in[1,\infty)$ and the weighted $L^p$ norm by $||.||_{p,\mu}$.\\

The Bergman projection $\BOmmu$ is a canonical object on the weighted space $(\Omega,\mu)$ and it is a fundamental question how perturbations of the domain $\Omega$ or the weight $\mu$ change the analytic properties of this canonical object. In this note, we are particularly interested in the following problem.\\

\begin{enumerate}
\item[] \textbf{$L^p$ Regularity Problem.} For a given domain $\Omega$ and a weight $\mu$ on $\Omega$, determine values of $p\in(1,\infty)$ such that the weighted Bergman projection $\BOmmu$ is bounded from $L^p(\Omega,\mu)$ to itself.\\ 
\end{enumerate}

Note that, by duality and self-adjointness, if $\BOmmu$ is bounded on $L^{p_0}(\Omega, \mu)$ for some $p_0>2$ then it is also bounded on $L^{q_0}(\Omega, \mu)$ where $\frac{1}{p_0}+\frac{1}{q_0}=1.$ Further, by interpolation, $\BOmmu$ is also bounded on $L^s(\Omega, \mu)$ for any $q_0\leq s \leq p_0$.

\subsection{Background} 
This problem is investigated in various forms in the literature. We mention a few results that motivate our work in this note.\\

For $\Omega=\mathbb{D}$ the unit disc in $\mathbb{C}^1$ and radial weights $\mu(z)=(1-|z|^2)^t$ for $t>-1$, the corresponding weighted Bergman projections are bounded on $L^p\left(\mathbb{D}, (1-|z|^2)^t\right)$ for any $p\in(1,\infty)$. This can be proven either by Schur's lemma (see \cite{ForelliRudin} or \cite{Zhubook}) or by singular integral theory (see \cite{McNeal94}). The same conclusion is also true for weights that are comparable to the weights above, see \cite{ZeytuncuComparabletoone} and \cite{ZeytuncuComparable}. On the other hand, in \cite{Dostanic04} there are examples of radial weights $\mu$ on $\mathbb{D}$ such that the weighted projections are bounded on $L^p(\mathbb{D},\mu)$ only if $p=2$.\\

In higher dimensions, \cite{PhongStein}, \cite{McNeal94}, \cite{McNSte94} and \cite{Charp06} contain some basic $L^p$ regularity results in the unweighted case. In these articles, it is shown that if $\Omega$ is a strongly pseudoconvex domain or a smoothly bounded convex domain of finite type in $\mathbb{C}^n$ or a smoothly bounded pseudoconvex domain of finite type in $\mathbb{C}^2$ or a decoupled domain in $\mathbb{C}^n$ then the ordinary   Bergman projection  $\mathbf{B}_{\Omega}$  is bounded from $L^p(\Omega)$ to $L^p(\Omega)$ for any $p\in(1,\infty).$ 
As for $L^p$ irregularity results in higher dimensions, \cite{Barrett84} and \cite{KrantzPeloso07Statements} contain the main examples. In \cite{Barrett84}, Barrett gives an example of a smoothly bounded non-pseudoconvex domain $D$ in $\mathbb{C}^2$ such that the ordinary projection $\textbf{B}_{D}$ is not bounded on $L^p(D)$ for $p\geq2+\frac{1}{k}$ where $k$ is a positive integer depending on the domain. In a recent series of papers \cite{KrantzPeloso07Statements}-\cite{KrantzPeloso07Proofs} Krantz and Peloso show that on the non-smooth worm domain $\mathcal{D}_{\beta}\subset \mathbb{C}^2$, the ordinary projection $\mathbf{B}_{\mathcal{D}_{\beta}}$ is bounded on $L^p(\mathcal{D}_{\beta})$ only if $p\in(\frac{2}{1+v_{\beta}},\frac{2}{1-v_{\beta}})$ where $v_{\beta}$ is determined by winding of the domain $\mathcal{D}_{\beta}$. Recently in \cite{BarrettSahutoglu}, the authors obtained irregularity results for the Bergman projections of some higher dimensional versions of worm domains in $\mathbb{C}^n$, $n\geq 3$.\\

\subsection{Outline and Results}
This paper consists of two parts. In the first part (Sections 2 and 3), we focus on the case $\Omega$ is the unit disc $\mathbb{D}$ in $\mathbb{C}^1$ and we vary the weight $\mu$ on $\mathbb{D}$. We investigate how $L^p$ mapping properties of weighted Bergman projections change as weights change on $\mathbb{D}$. In the second part (Section 4), we focus on ordinary Bergman projections of higher dimensional domains i.e. the weight is fixed to Lebesgue measure and the domain is perturbed. In this part, we apply Forelli-Rudin's inflation idea  to weighted examples in the first part to construct domains in $\mathbb{C}^2$ whose ordinary Bergman projections exhibit irregularities in $L^p$ scale. The following two theorems formulate weighted results.\\
\begin{theorem}\label{Type2mainIntroduction}
If $\lambda$ is a radial weight on $\mathbb{D}$ which satisfies \symbolfootnote[2]{Here, we 
abuse the notation and consider $\lambda$ as a function on $[0,1]$ and by $\lambda(z)$ we mean $\lambda(|z|)$.}
\begin{enumerate}
\item $\lambda(r)$ is a smooth function on $[0,1]$,
\item $\lambda^{\left(n\right)}(1):=\frac{d^n}{dr^n}\lambda(1)=0$ for any $n\in\mathbb{N}$,
\item for any $n\in\mathbb{N}$ there exists $a_n\in(0,1)$ such that $(-1)^n\lambda^{\left(n\right)}(r)$ is non-negative on the interval $(a_n,1)$.
\end{enumerate}
Then the weighted Bergman projection $\mathbf{B}_{\lambda}$ is bounded from $L^p(\mathbb{D}, \lambda)$ to $L^p(\mathbb{D}, \lambda)$ only for $p=2$.
\end{theorem}

The conditions in Theorem \ref{Type2mainIntroduction} can be checked for particular weights and we do this in the corollary following the proof in the second section. In particular, we recover and extend the result in \cite{Dostanic04}. The proof uses successive integration by parts to compute the asymptotics of the moment function of weight $\lambda$. The infinite order vanishing of $\lambda$ is crucial to integrate by parts infinitely many times.\\
 
\begin{theorem}\label{One}
For any given $p_0>2$ there exists a weight $\mu_0$ on $\mathbb{D}$ such that the weighted projection $\mathbf{B}_{\mu_0}$ is bounded on $L^p(\mathbb{D}, \mu_0)$ only if $p\in(q_0,p_0)$, where $\frac{1}{p_0}+\frac{1}{q_0}=1$.
\end{theorem}
This theorem is the first appearance of weights of this type. The proof is constructive and weights are explicitly written down. One key ingredient of the proof is the Bekoll\'e-Bonami condition. Lanzani and Stein present a clear explanation of this condition in \cite{LanzaniStein04}.\\

In Section 4, by using the weighted results on $\mathbb{D}$, we construct domains in $\mathbb{C}^2$ with irregular ordinary Bergman projections. The following theorems formulate these constructions.\\
\begin{theorem}\label{Two}
There are bounded domains $\Omega$ in $\mathbb{C}^2$ such that the ordinary Bergman projections of these domains $\Omega$ are bounded on $L^p(\Omega)$ only for $p=2$.
\end{theorem}
The remarks at the end of Barrett's paper \cite{Barrett84} contain an example of a similar domain that is smoothly bounded but not complete Hartogs. The domains we construct here are even Reinhardt but do not have smooth boundary. 

Additionally, the domains in Theorem \ref{Two} are simply connected. This highlights one more difference between one complex variable and several complex variables. In \cite{LanzaniStein04} and \cite{Hedenmalm02}, it is shown that there exists a universal constant $r>2$ such that the ordinary Bergman projection of any simply connected proper domain $D$ in $\mathbb{C}^1$ is bounded from $L^p(D)$ to $L^p(D)$ at least for any $p\in(r',r)$, where $\frac{1}{r'}+\frac{1}{r}=1$. We see that in $\mathbb{C}^n$ for $n\geq 2$ there exists no such a universal constant.\\

\begin{theorem}\label{Three}
For any given $p_0>2$, there is a bounded domain $\Omega_0$ in $\mathbb{C}^2$ such that the ordinary Bergman projection $\mathbf{B}_{\Omega_0}$ is bounded on $L^p(\Omega_0)$ only if $p\in(q_0,p_0)$, where $\frac{1}{p_0}+\frac{1}{q_0}=1$.
\end{theorem}
The main difference between this theorem and other $L^p$ irregularity results in the literature is the regularity part of the statement. Namely, we not only prove unboundedness but also prove that Bergman projection is bounded for a certain range.\\

The content of this paper is a part of my PhD dissertation at The Ohio State University. 
I thank J.D. McNeal, my advisor, for introducing me to this field and helping me with various points. 
I thank K. Koenig for helpful suggestions during this project.
I also thank the anonymous referee for helpful recommendations to improve the presentation of the paper.  \\

\section{Proof of Theorem \ref{Type2mainIntroduction}}
The first examples of weights of kind in Theorem \ref{Type2mainIntroduction} appear in \cite{Dostanic04}. 
Before the proof of Theorem \ref{Type2mainIntroduction}, we present the following corollary 
to give explicit examples of weights satisfying the properties listed in the theorem.

\begin{corollary}\label{DostanicCorollary}
Let
\begin{equation}\label{Dostanic}
\lambda(r)=(1-r^2)^A\exp\left(\frac{-B}{(1-r^2)^{\alpha}}\right)
\end{equation}
for some $A\geq 0, B>0, \alpha>0$. Then $\lambda$ satisfies the conditions listed in Theorem \ref{Type2mainIntroduction} and
$\mathbf{B}_{\lambda}$ is bounded from $L^p(\lambda)$ to $L^p(\lambda)$ only for $p=2$ and unbounded for $p\in(1,2)$.
\end{corollary}
\noindent The claim of the corollary was first proven in \cite{Dostanic04} with the restriction $0<\alpha\leq 1$.\\

\begin{proof}
We have to check the functions defined by \eqref{Dostanic} satisfy the properties in Theorem \ref{Type2mainIntroduction}. The first two conditions follow immediately and the last one can be seen by a careful look at the successive derivatives. We do this here only for the special case $A=0, B=1, \alpha=1$ and the general case follows similarly.
We have
\begin{align*}
\lambda(r)&=\exp\left(\frac{-1}{1-r^2}\right),\\
\lambda'(r)&=\left(\frac{-2r}{(1-r^2)^2}\right)\exp\left(\frac{-1}{1-r^2}\right),\\
\vdots \\
\lambda^{(n)}(r)&=\left(\frac{(-2r)^n}{(1-r^2)^{2n}}+ \text{lower order terms} \right)\exp\left(\frac{-1}{1-r^2}\right).
\end{align*}
As $r$ gets closer to $1$, the dominant term in the parenthesis is $\frac{(-2r)^n}{(1-r^2)^{2n}}$ and this term satisfies the  third condition.\\
\end{proof}

\begin{proof}[Proof of Theorem \ref{Type2mainIntroduction}]
It is clear that the weighted projection $\wB$ is bounded for $p=2$, so in order to prove the theorem, we have to show unboundedness for $1<p<2$.

\noindent \textit{Step One.} Analyze the moment function $\Phi(x)=\int_0^1r^{2x+1}\lambda(r)dr$, for $x\geq 0$.\\

For any $n>0$ we integrate by parts to obtain
\begin{align*}
\Phi(x)=\int_0^1r^{2x+1}\lambda(r)dr&=\frac{1}{2x+2}\int_0^1r^{2x+2}(-1)\lambda^{(1)}(r)dr\\
&=\vdots\\
&=\frac{1}{2x+2}\dots\frac{1}{2x+1+n}\int_0^1r^{2x+1+n}(-1)^n\lambda^{(n)}(r)dr.
\end{align*}
For convenience we use notation $\psi_n(r)=(-1)^n\lambda^{(n)}(r)$ and $\Phi_n(x)=\int_0^1r^{2x+1+n}\psi_n(r)dr$. Therefore, for any $n>0$
\begin{equation}\label{nrelation}
\Phi(x)=\frac{1}{2x+2}\dots\frac{1}{2x+1+n}\Phi_n(x).
\end{equation}

\noindent At this stage we need the third condition of the theorem because we do not know if $\Phi_n(x)$ is log-convex. Since $\psi_n(r)$ is not necessarily non-negative on $(0,1)$ we cannot use H\"older's inequality. Fortunately, we know that $\psi_n(r)$ is non-negative on $(a_n,1)$ and for large values of $x$ two integrals $\int_0^1r^{2x+1+n}\psi_n(r)dr$ and $\int_{a_n}^1r^{2x+1+n}\psi_n(r)dr$ are almost the same.

To make this point rigorous, we define
{\allowdisplaybreaks
\begin{equation}\label{tildemoment}
\widetilde{\Phi}_n(x)=\int_{a_n}^1r^{2x+1+n}\psi_n(r)dr.
\end{equation}
Note that
\begin{align*}
\left|\frac{\Phi_n(x)}{\widetilde{\Phi}_n(x)}-1\right|&=\left|\frac{\int^{a_n}_0r^{2x+1+n}\psi_n(r)dr}{\int_{a_n}^1r^{2x+1+n}\psi_n(r)dr}\right|\\&\leq \frac{\int^{a_n}_0r^{2x+1+n}|\psi_n(r)|dr}{\int_{a_n}^1r^{2x+1+n}\psi_n(r)dr}\\&\leq \max_{0\leq s \leq a_n}|\psi_n(s)|\frac{\int^{a_n}_0r^{2x+1+n}dr}{a_n^{2x+1+n}\int_{a_n}^1\psi_n(r)dr}\\&=\frac{\max_{0\leq s \leq a_n}|\psi_n(s)|}{\int_{a_n}^1\psi_n(r)dr}\frac{\frac{a_n^{2x+2+n}}{2x+2+n}}{a_n^{2x+1+n}}\\
&=\frac{\max_{0\leq s \leq a_n}|\psi_n(s)|}{\int_{a_n}^1\psi_n(r)dr}\frac{a_n}{2x+2+n}\\&=C(n)\frac{1}{2x+2+n}.
\end{align*}
Thus, for any $n>0$
\begin{equation}\label{limit}
\lim_{x\to\infty}\frac{\Phi_n(x)}{\widetilde{\Phi}_n(x)}=1.
\end{equation}}

\noindent If we combine \eqref{nrelation} and \eqref{limit} we get for any $n>0$ there exists $X(n)$ such that for any $x>X(n)$ we have
\begin{equation}\label{Theta}
\frac{1}{2}\left(\frac{1}{2x+2}\dots\frac{1}{2x+1+n}\widetilde{\Phi}_n(x)\right)\leq \Phi(x)\leq 2\left(\frac{1}{2x+2}\dots\frac{1}{2x+1+n}\widetilde{\Phi}_n(x)\right).
\end{equation}
Again, for convenience  label $\Theta_n(x)=\frac{1}{2x+2}\dots\frac{1}{2x+1+n}\widetilde{\Phi}_n(x)$ and write $\Theta_n(x)=e^{-\theta_n(x)}$.

\noindent We also note that by H\"older's inequality, for any $0<t<1$ and $x,y>0$, we have
\begin{align*}
\widetilde{\Phi}_n(tx+(1-t)y)&=\int_{a_n}^1r^{2tx+2(1-t)y+1+n}\psi_n(r)dr\\
&=\int_{a_n}^1\left(r^{2x+1+n}\psi_n(r)\right)^t\left(r^{2y+1+n}\psi_n(r)\right)^{1-t}dr\\
&\leq\left(\int_{a_n}^1r^{2x+1+n}\psi_n(r)dr\right)^t\left(\int_{a_n}^1r^{2x+1+n}\psi_n(r)dr\right)^{1-t}\\
&=\left(\widetilde{\Phi}_n(x)\right)^t\left(\widetilde{\Phi}_n(y)\right)^{1-t}.
\end{align*}
Thus, $\log\widetilde{\Phi}_n(x)$ is convex.
\vskip 1 cm
\noindent \textit{Step Two.} A specific sequence of functions.

\noindent We take $k,m\in \mathbb{N}$ and consider the action of $\wB$ on functions $z^{km}\bar{z}^m$. A simple calculation shows that
\begin{equation*}
\wB(z^{km}\bar{z}^m)=\frac{\Phi(km)}{\Phi\left((k-1)m\right)}z^{(k-1)m}.
\end{equation*}
Now we focus on the following ratio
\begin{align*}
R_k(m)&=\frac{||\wB(z^{km}\bar{z}^m)||_{p,\lambda}^p}{||z^{km}\bar{z}^m||_{p,\lambda}^p}\\
&=\left(\frac{\Phi(km)}{\Phi\left((k-1)m\right)}\right)^p\frac{||z^{(k-1)m}||_{p,\lambda}^p}{||z^{km}\bar{z}^m||_{p,\lambda}^p}\\
&=\left(\frac{\Phi(km)}{\Phi\left((k-1)m\right)}\right)^p\frac{\Phi(\frac{p}{2}(k-1)m)}{\Phi(\frac{p}{2}(k+1)m)}.
\end{align*}
Given $1<p<2$, fix $k>\frac{2+p}{2-p}$ independent of $m$. This choice of $k$ gives us the following inequalities
\begin{equation}\label{conf}
\frac{p}{2}(k-1)m<\frac{p}{2}(k+1)m<(k-1)m<km.
\end{equation}

\noindent By \eqref{Theta} we know that for any $n>0$ there exists $M(n)$ such that for all $m>M(n)$ we have
\begin{align*}
R_k(m)&\geq C(p)\left(\frac{\Theta_n(km)}{\Theta_n\left((k-1)m\right)}\right)^p\frac{\Theta_n(\frac{p}{2}(k-1)m)}{\Theta_n(\frac{p}{2}(k+1)m)}\\
&=C(p)\exp \left[p\theta_n\left((k-1)m\right)-p\theta_n\left(km\right)+\theta_n\left(\frac{p}{2}(k+1)m\right)-\theta_n\left(\frac{p}{2}(k-1)m\right)\right]\\
&=C(p)\exp\left[ -pm\theta_n'(v_m)+pm\theta_n'(w_m)\right]\\
&=C(p)\exp \left[ -pm(v_m-w_m)\theta_n''(u_m)\right]
\end{align*}
where we used the mean value theorem twice and $w_m\in(\frac{p}{2}(k-1)m,\frac{p}{2}(k+1)m)$, $v_m\in((k-1)m,km)$ and $u_m\in(w_m,v_m)$. Further note that $w_m,v_m,u_m$ and $v_m-w_m$ are comparable to $m$, where the comparison constants only depend on $p$ (they also depend on $k$ but recall $k$ is fixed and it depends on $p$). Therefore, for all $m>M(n)$
\begin{equation}\label{ratiolower}
R_k(m)\geq C(p)\exp[-D(p)u_m^2\theta_n''(u_m)]
\end{equation}
where $C(p)$ and $D(p)$ are strictly positive constants that depend only on $p$.
\vskip 1cm
\noindent \textit{Step Three.} Second derivative of $\theta_n(x)$.

\noindent Now we consider $\theta_n''(x)$, for any $n>0$. We have
\begin{align*}
\theta_n(x)&=-\log \Theta_n(x)\\
&=\log(2x+2)+\dots+\log(2x+1+n)-\log\widetilde{\Phi}_n(x).
\end{align*}
This implies
\begin{align*}
-x^2\theta_n''(x)&=4\left(\frac{x^2}{(2x+2)^2}+\dots+\frac{x^2}{(2x+1+n)^2}\right)+x^2[\log\widetilde{\Phi}_n(x)]''.
\end{align*}
Note that $\frac{x^2}{(2x+1+n)^2}\geq\frac{1}{8}$ for sufficiently large $x$ (for fixed $n$), and if we use the fact that $\log\widetilde{\Phi}_n(x)$ is convex then for any $n>0$ there exists $X(n)>0$ such that for all $x>X(n)$ we have
\begin{equation}\label{estimate}
-x^2\theta_n''(x)\geq 4\left(\frac{1}{8}+\dots+\frac{1}{8}\right)=\frac{n}{2}.
\end{equation}

\noindent If we combine this estimate with the inequality \eqref{ratiolower} we obtain that for any $n>0$ there exists $M(n,p)$ such that for all $m>M(n,p)$ we have
$$R_k(m)\geq C(p)\exp [D(p)\frac{n}{2}].$$
This certainly implies that for any $1<p<2$ $$\lim_{m\to \infty}R_k(m)=\infty$$
and this concludes the proof.\\
\end{proof}

Note that if $p=2$ then we can not have \eqref{conf} and further one can easily see that $R(k,m)\leq 1$ for any $k,m>0$ by H\"older's inequality. This complies with the fact that $\wB$ is bounded if $p=2$.\\

6

\section{Proof of Theorem \ref{One}}
In this section, we prove Theorem \ref{One}. We use a theorem of Bekoll{\'e}-Bonami which was announced in \cite {BekolleBonami78} and explained well in \cite{LanzaniStein04} and \cite{Borichev04}. This theorem is similar to Muckenhoupt's $A_p$ condition for the Hilbert transform. For $A_p$ weights, see \cite{Muckenhoupt72} and \cite{CoifmanFefferman74}.

One point to distinguish between the discussion here and the classification result of Bekoll{\'e}-Bonami is that in our work the projection operator changes as the weight changes. However, in \cite{BekolleBonami78} or \cite{LanzaniStein04}, the (ordinary) projection operator is fixed and only the function spaces change as the weight changes.\\

We start with copying the following definition and the theorem from \cite{LanzaniStein04}. We use $\uplane$ to denote the upper half plane $\{z\in\mathbb{C}^1~:~\Im (z)>0\}$.

\begin{definition}\label{def}
A weight $\mu$ on $\mathbb{R}^2_+$ is said to be in class $A_p^+(\mathbb{R}^2_+)$ if there exists $C>0$ such that for any disc $D=D(x_0,R)$, where $x_0\in\mathbb{R}$ and $R>0$, we have
\begin{equation}\label{Apinequality}
\frac{1}{|D\cap\mathbb{R}^2_+|^p}\int_{D\cap\mathbb{R}^2_+}\mu(z)dA(z)\left(\int_{D\cap\mathbb{R}^2_+}\mu(z)^{\frac{1}{1-p}}dA(z)\right)^{p-1}\leq C.
\end{equation}
Here $|.|$ denotes standard Lebesgue measure.
\end{definition}

\begin{theorem}\label{LanzaniStein}{\textnormal{(Bekoll\'e \& Bonami - Lanzani \& Stein)}}
Let $\mu$ be a weight on $\uplane$ then $\mathbf{P_1}$\symbolfootnote[2]{In this section, $\mathbf{P_1}$ denotes the ordinary Bergman projection on $\uplane$.} is bounded from $L^p(\mathbb{R}^2_+, \mu)$ to $L^p(\mathbb{R}^2_+, \mu)$ if and only if
$\mu\in A_p^+(\mathbb{R}^2_+).$
\end{theorem}
\begin{proof} See \cite[Proposition 4.5]{LanzaniStein04}.\\ \end{proof}

The goal in the present section is to relate this result to weighted Bergman projections. Unfortunately, it is not simple to do this for an arbitrary weight $\mu$ since it is not simple to relate the weighted Bergman projection $\Bmu$ and the ordinary Bergman projection  $\mathbf{B_1}$ in general. But the good news is that this relation is possible if we focus on weights $\mu$ of the form $\mu=|g|^2$ for a non-vanishing holomorphic function $g$ on $\mathbb{D}$. The following theorem expresses this idea.\\

For the rest of this section, let $\phi(\zeta)=\frac{i-\zeta}{i+\zeta}$ be the biholomorphism from $\mathbb{R}^2_+$ to $\mathbb{D}$, and $\psi$ be the inverse of $\phi.$

\begin{theorem}\label{Type3}
Let $g$ be a holomorphic function on $\mathbb{D}$ which does not vanish inside $\mathbb{D}$. Let $\omega=|g|^2$ be the weight and $p\in (1,\infty)$. Then the following are equivalent
\begin{enumerate}
\item $\mathbf{B_{|g|^2}}$ is bounded from $L^p(\mathbb{D}, |g|^2)$ to $L^p(\mathbb{D}, |g|^2)$,
\item $\mathbf{B_1}$ is bounded from $L^p(\mathbb{D}, |g|^{2-p})$ to $L^p(\mathbb{D}, |g|^{2-p})$,
\item $\mathbf{P_1}$ is bounded from $L^p(\mathbb{R}^2_+, |(g\circ\phi)\phi'|^{2-p})$ to $L^p(\mathbb{R}^2_+, |(g\circ\phi)\phi'|^{2-p})$,
\item $|(g\circ\phi)\phi'|^{2-p}\in A_p^+(\mathbb{R}^2_+).$
\end{enumerate}
\end{theorem}

\begin{proof}
\noindent The equivalence of $(3)$ and $(4)$ is nothing but Theorem \ref{LanzaniStein}.

Let's look at the equivalence of $(2)$ and $(3)$. A more general form of this equivalence can be proved. Namely, if $\mu$ is a weight on $\mathbb{D}$ then the following two are equivalent
\begin{itemize}
\item $\mathbf{B_1}$ is bounded on $L^p(\mathbb{D}, \mu)$,
\item $\mathbf{P_1}$ is bounded on $L^p(\mathbb{R}^2_+, (\mu\circ\phi)|\phi'|^{2-p})$.
\end{itemize}
Put $\Lambda(\nu)=(\mu\circ\phi(\nu))|\phi'(\nu)|^{2-p}$, and let $B_1(z,w)$ and $P_1(\zeta,\nu)$ be the ordinary Bergman kernels on $\mathbb{D}$ and $\uplane$, respectively. The following transformation formulas are well known:
\begin{align*}
P_1(\zeta,\nu)&=\phi'(\zeta)~B_1\left(\phi(\zeta),\phi(\nu)\right)~\overline{\phi'(\nu)},\\
B_1(z,w)&=\psi'(z)~P_1\left(\psi(z),\psi(w)\right)~\overline{\psi'(w)}.
\end{align*}
Take $f\in L^p(\mathbb{R}^2_+, \Lambda)$ and change variables to obtain
\begin{align*}
\int_{\uplane}|f(\nu)|^p\Lambda(\nu)dA(\nu)&=\int_{\uplane}|f(\nu)|^p(\mu\circ\phi(\nu))|\phi'(\nu)|^{2-p}dA(\nu)\\
&=\int_{\mathbb{D}}|f(\psi(w))|^p|\psi'(w)|^p\mu(w)dA(w).
\end{align*}
Thus, $(f\circ\psi)\psi'\in L^p(\mathbb{D}, \mu).$ Now we consider the action of projection operators
\begin{align*}
\mathbf{B_1}\left[ f(\psi(w)\psi'(w)\right](z)&=\int_{\mathbb{D}}B_1(z,w)f(\psi(w))\psi'(w)dA(w)\\
&=\int_{\uplane}B_1(z,\phi(\nu))f(\nu)\overline{\phi'(\nu)} dA(w).
\end{align*}
This implies
\begin{align*}
\mathbf{B_1}\left[ f(\psi(w))\psi'(w)\right](\phi(\zeta))&=\int_{\uplane}B_1(\phi(\zeta),\phi(\nu))f(\nu)\overline{\phi'(\nu)} dA(w)\\
&=\frac{1}{\phi'(\zeta)}\int_{\uplane}P_1(\zeta,\nu)f(\nu)dA(\nu)\\
&=\frac{1}{\phi'(\zeta)}\mathbf{P_1}f(\zeta).
\end{align*}

We assume that $\mathbf{B_1}$ is bounded on $L^p(\mathbb{D},\mu)$ and prove that $\mathbf{P_1}$ is bounded on $L^p(\uplane,\Lambda)$ as follows
\begin{align*}
||\mathbf{P_1}f||_{p,\Lambda}^p&=\int_{\uplane}\left|\mathbf{P_1}f(\zeta)\right|^p\Lambda(\zeta)dA(\zeta)=
\int_{\uplane}\left|\mathbf{P_1}f(\zeta)\right|^p(\mu\circ\phi(\zeta))|\phi'(\zeta)|^{2-p}dA(\zeta)\\&=\int_{\uplane}|\phi'(\zeta)|^p\left| \mathbf{B_1}\left[ f(\psi) \psi' \right](\phi(\zeta))\right|^p(\mu\circ\phi(\zeta))|\phi'(\zeta)|^{2-p}dA(\zeta)\\&=\int_{\mathbb{D}}\left| \mathbf{B_1}\left[ f(\psi) \psi' \right](z)\right|^p\mu(z)dA(z)~\text{ (use boundedness)}\\&\leq C \int_{\mathbb{D}}\left| f(\psi) \psi' \right|^p\mu(z)dA(z)\\&=C\int_{\uplane}|f(\zeta)|^p(\mu\circ\phi(\zeta))|\phi'(\zeta)|^{2-p}dA(\zeta)\\&\leq C||f||_{p,\Lambda}^p.
\end{align*}
The same arguments above similarly prove that boundedness of $\mathbf{P_1}$ implies boundedness of $\mathbf{B_1}$. Therefore, we finish the proof of the equivalence of $(2)$ and $(3)$.
\vskip 1cm
Next, we prove the equivalence of $(1)$ and $(2)$. We start with an identity between the kernels. Let $B_{\omega}(z,w)$ be the weighted Bergman kernel ($\omega=|g|^2$). By using the orthonormal representation for the kernel we obtain
\begin{equation}\label{kernelrelation}
g(z)B_{\omega}(z,w)\overline{g(w)}=B_1(z,w).
\end{equation}
Indeed, if $\{e_n(z)\}$ is an orthonormal basis for $L_a^2(1)$, then $\{\frac{e_n(z)}{g(z)}\}$ is an orthonormal basis for $L_a^2(\omega)$.

\noindent By using this relation between the kernels we obtain the following relation between the operators
\begin{equation}\label{operatorrelation}
g(z)(\mathbf{B_{\omega}}f)(z)=(\mathbf{B}_1(f.g))(z) \quad \text{ for }f\in L^2(\omega).
\end{equation}

\noindent Suppose (2) is true. Then
\begin{align*}
||\mathbf{B_{\omega}}f||^p_{p,\omega}&=\int_{\mathbb{D}}|(\mathbf{B_{\omega}}f)(z)|^p|g(z)|^2&\\
&=\int_{\mathbb{D}}|(\mathbf{B}_1(f.g))(z)|^p|g(z)|^{2-p}=||\mathbf{B}_1(f.g)||^p_{p,|g|^{2-p}}&\\
&\lesssim||f.g||^p_{p,|g|^{2-p}}=||f||^p_{p,\omega}&
\end{align*}
and (1) follows.

\noindent Now suppose (1) is true. Then
\begin{align*}
||\mathbf{B}_1f||^p_{p,|g|^{2-p}}&=\int_{\mathbb{D}}|(\mathbf{B}_1f)(z)|^p|g(z)|^{2-p}&\\
&=\int_{\mathbb{D}}|(\mathbf{B}_{\omega}(f/g))(z)|^p|g(z)|^{2}=||\mathbf{B}_{\omega}(\frac{f}{g})||^p_{p,\omega}&\\
&\lesssim||\frac{f}{g}||^p_{p,\omega}=||f||^p_{p,|g|^{2-p}}&
\end{align*}
and (2) follows. This finishes the proof of the equivalence of $(1)$ and $(2)$.\\
\end{proof}
\begin{remark}
Absence of a relation of the form \eqref{operatorrelation} for an arbitrary weight $\mu$ is the main difficulty to generalize Theorem \ref{Type3} to larger classes of weights.
\end{remark}
\noindent Just for clarity, we rewrite the first and the last condition in Theorem \ref{Type3} as follows.
\begin{corollary}\label{F}
Let $F$ be a non-vanishing holomorphic function on $\mathbb{R}^2_+$ and let $\omega(z)=|(F\circ\psi(z))\psi'(z)|^2$ then $\Bom$ is bounded on $L^p(\mathbb{D}, \omega)$ if and only if $|F|^{2-p}\in A^+_p(\mathbb{R}^2_+).$\\
\end{corollary}

\noindent The next corollary gives explicit examples of weights in Theorem \ref{One}.

\begin{corollary}\label{Type3Corollary}
Let $F(\zeta)=\zeta^{2/3}$ for $\zeta \in \mathbb{R}^2_+$ and $\omega=|(F\circ\psi(z))\psi'(z)|^2$. The weighted projection $\mathbf{B_{\omega}}$ is bounded on $L^p(\omega)$ for $p\in (\frac{5}{4},5)$ and unbounded for any other values of $p$.
\end{corollary}
It is easy to see that the exponent $\frac{2}{3}$ is not special. We can generalize the corollary so that for any given $p_0>2$ we can find a weight function $\omega_0$ (take $F(\zeta)=\zeta^{\frac{2}{p_0-2}}$) for which the boundedness range is exactly $(q_0,p_0)$. This proves Theorem \ref{One} stated in Introduction.\\

\begin{proof}
By Corollary \ref{F}, we need to check for which values of $p$, $$|\zeta|^{\frac{2}{3}(2-p)}\in\Aptt.$$
We start with $p\geq 5$. In this case, $\frac{2}{3}(2-p)=-2-2\epsilon$ for some $\epsilon\geq0$. Also we take $\mathbf{D}_1=D(0,1)\cap \uplane$ then
\begin{align*}
\int_{\mathbf{D}_1}|\zeta|^{\frac{2}{3}(2-p)}dA(\zeta)=c\int_0^1r^{-2-2\epsilon}rdr=\infty.
\end{align*}
This shows that the $\Aptt$ inequality fails and $|\zeta|^{\frac{2}{3}(2-p)}\not\in\Aptt$ for $p\geq5$. Consequently, $\Bom$ is unbounded on $L^p(\omega)$ for $p\geq5$.\\
The next step is $2\leq p<5$. In this case, $\frac{2}{3}(2-p)=-2+2\epsilon$ for some $\epsilon>0$. Given any $\mathbf{D}_2=D(x_0,R)$, for $x_0\in \mathbb{R}$ and $R>0$. There are two possibilities: either $\mathbf{D}_2\cap D(0,2R)$ is empty or not.

\noindent Suppose $\mathbf{D}_2 \cap D(0,2R)$ is not empty; then clearly $\mathbf{D}_2\subset D(0,4R)$, and
\begin{align*}
&\frac{1}{|\mathbf{D}_2\cap\mathbb{R}^2_+|^p}\int_{\mathbf{D}_2\cap\mathbb{R}^2_+}|F(z)|^{2-p}dA(z)
\left(\int_{\mathbf{D}_2\cap\mathbb{R}^2_+}|F(z)|^{\frac{2-p}{1-p}}dA(z)\right)^{p-1}\\ 
&\leq \frac{c}{R^{2p}}\int_{D(0,4R)\cap\mathbb{R}^2_+}|F(z)|^{2-p}dA(z)
\left(\int_{D(0,4R)\cap\mathbb{R}^2_+}|F(z)|^{\frac{2-p}{1-p}}dA(z)\right)^{p-1}\\
&=\frac{c}{R^{2p}}\int_0^{4R}r^{-2+2\epsilon}rdr
\left(\int_0^{4R} r^{\frac{2-2\epsilon}{4-3\epsilon}}rdr\right)^{4-3\epsilon}\\
&=\frac{c}{R^{2p}}R^{2\epsilon}R^{10-8\epsilon}\\
&=c.
\end{align*}
This implies the supremum over discs of this type is finite.

\noindent Suppose $\mathbf{D}_2 \cap D(0,2R)$ is empty; then clearly $|z|\sim |x_0|$ for any $z\in \mathbf{D}_2$, and
\begin{align*}
&\frac{1}{|\mathbf{D}_2\cap\mathbb{R}^2_+|^p}\int_{\mathbf{D}_2\cap\mathbb{R}^2_+}|F(z)|^{2-p}dA(z)
\left(\int_{\mathbf{D}_2\cap\mathbb{R}^2_+}|F(z)|^{\frac{2-p}{1-p}}dA(z)\right)^{p-1}\\ 
&\leq \frac{c}{R^{2p}}\int_{\mathbf{D}_2\cap\mathbb{R}^2_+}|x_0|^{\frac{2}{3}(2-p)}dA(z)
\left(\int_{\mathbf{D}_2\cap\mathbb{R}^2_+}|x_0|^{\frac{2(2-p)}{3(1-p)}}dA(z)\right)^{p-1}\\
&=\frac{c}{R^{2p}}|\mathbf{D}_2||x_0|^{-2+2\epsilon}|\mathbf{D}_2|^{4-3\epsilon}|x_0|^{2-2\epsilon}\\
&=c.
\end{align*}
This again implies that the supremum over discs of this type is finite. These two cases show that
for $2\leq p<5$, $|\zeta|^{2-p}\in \Aptt$. Consequently, $\Bom$ is bounded on $L^p(\omega)$ for $2\leq p<5$.\\

\noindent This with the duality and the self adjointness of $\Bom$ finish the proof of Corollary \ref{Type3Corollary}.
\end{proof}

\begin{remark}\label{remarkF}
The weight $\omega$ in Corollary \ref{Type3Corollary} is unbounded on $\mathbb{D}$. But if we take $F(\zeta)=\frac{-2i}{(i+\zeta)^2}\left(\frac{-2\zeta}{i+\zeta}\right)^{2/3}$ then $\omega(z)=|z-1|^{4/3}$ is a bounded function on $\mathbb{D}$ and the conclusion of Corollary \ref{Type3Corollary} holds for this choice, too. The proof works the same way. See Appendix A for details.\\
\end{remark}

\section{Domains with Irregular Bergman Projections}

In this section, we lift up the results of the previous chapters to $\mathbb{C}^2.$ For a given weight $\mu$ on $\mathbb{D}$ we define the following domain in $\mathbb{C}^2$:
\begin{equation}\label{C2domain}
\Omega=\{(z,w)\in \mathbb{C}^2~|~ z\in\mathbb{D},~  |w|^2<\mu(z)\}.
\end{equation}
Let $\mathbf{B}_{\Omega}$ be the ordinary Bergman projection of $\Omega$:
\begin{align*}
\mathbf{B}_{\Omega}&:~L^2(\Omega)\to L^2_a(\Omega),\\
\left(\mathbf{B}_{\Omega}F\right)(z,w)&=\int_{\Omega}B_{\Omega}\left[(z,w),(t,s)\right]F(t,s)dV(t,s).
\end{align*}

\begin{proposition}\label{inflation}
We have the following relation  between the kernels 
\begin{equation}
B_{\Omega}\left[(z,w),(t,s)\right]=\frac{1}{2\pi}\sum_{m=0}^{\infty}(2m+2)w^mK_m(z,t)\overline{s}^m
\end{equation}
where $K_m(z,t)$ is the weighted Bergman kernel for the weight $\mu^{m+1}$ on $\mathbb{D}$. 
\end{proposition}
\begin{proof}
See \cite{ForelliRudin}, \cite{Ligocka89} or \cite{BoasFuStraube99}. This relation is sometimes called the Forelli-Rudin formula or inflation principle.
\end{proof}
In particular, $\pi B_{\Omega}\left[(z,0),(t,0)\right]=K_0(z,t)=B_{\mu}(z,t)$ in our earlier notation.\\ 

This relation between the kernels can be used to relate the $L^p$ mapping properties of the projections. 

\begin{proposition}\label{firsthalf}
For a given $p\in(1,\infty)$, suppose that $\Bmu$ is unbounded on $L^p(\mathbb{D},\mu)$ then $\mathbf{B}_{\Omega}$ is also unbounded on $L^p(\Omega).$
\end{proposition}

\begin{proof}
Unboundedness of $\Bmu$ on $L^p(\mathbb{D},\mu)$ implies that there exists a sequence of functions $\{f_n(z)\}$ in $L^p(\mu)$ such that the ratio $$\frac{||\Bmu f_n||^p_{p,\mu}}{||f_n||^p_{p,\mu}}$$
is unbounded. Define $F_n(z,w)=f_n(z)$. Clearly
$$f_n\in L^p(\mu)~\implies~F_n\in L^p({\Omega}) ~\text{ and }~ ||f_n||^p_{p,\mu}=\pi||F_n||^p_{p,\Omega}.$$
The projections of $F_n$ and $f_n$ are related as
\begin{align*}
\mathbf{B}_{\Omega}F_n(z,0)&=\int_{\Omega}B_{\Omega}\left[(z,0),(t,s)\right]F_n(t,s)dV(t,s)\\
&=\int_{\Omega}B_{\Omega}\left[(z,0),(t,s)\right]f_n(t)dV(t,s)\\
&=\int_{\mathbb{D}}f_n(t)\int_{|s|^2<\mu(t)}B_{\Omega}\left[(z,0),(t,s)\right]dA(s)dA(t)\\
&=\int_{\mathbb{D}}f_n(t)cB_{\Omega}\left[(z,0),(t,0)\right]dA(t)\\
&=\int_{\mathbb{D}}f_n(t)cB_{\mu}(z,t)dA(t)\\
&=c\mathbf{B}_{\mu}f_n(z).
\end{align*}
Here we use the fact that $B_{\Omega}\left[(z,0),(t,s)\right]$ is anti-holomorphic in $s$ therefore the mean value property holds in $s$. In order to compare the $L^p$ norms of the projections we argue as follows

\begin{align*}
||\mathbf{B}_{\Omega}F_n||_{p,\Omega}^p&=\int_{\Omega}|\mathbf{B}_{\Omega}F_n(z,w)|^pdV(z,w)\\
&=\int_{\mathbb{D}}\int_{|w|^2<\mu(z)}|\mathbf{B}_{\Omega}F_n(z,w)|^pdA(w)dA(z)\\
&\geq \int_{\mathbb{D}}|\mathbf{B}_{\Omega}F_n(z,0)|^p\mu(z)dA(z)~\text{ by the sub-mean value property}\\
&=c \int_{\mathbb{D}}|\mathbf{B}_{\mu}f_n(z)|^p\mu(z)dA(z)~\text{ by the identity above}\\
&=c ||\mathbf{B}_{\mu}f_n||^p_{p,\mu}.
\end{align*}
Therefore, the ratio $$\frac{||\mathbf{B}_{\Omega}F_n||^p_{p,\Omega}}{||F_n||^p_{p,\Omega}}$$
is unbounded, too. This finishes the proof.\\
\end{proof}

We do not know if the converse of this theorem is true in general. Although, constructing a \textit{bad sequence} of functions on $\Omega$ from the one on $(\mathbb{D},\mu)$ works fine, we do not know how to control all the projections on $\Omega$ by just the projections on $(\mathbb{D},\mu)$. Nevertheless, again there is a certain class of $\mu$ for which we can prove the converse.

\begin{proposition}\label{secondhalf}
Let $g$ be a holomorphic function on $\mathbb{D}$ which does not vanish inside $\mathbb{D}$ and let $\omega=|g|^2$ be the weight. Suppose $\Bom$ is bounded on $L^{p}(\omega)$ for some $p\in(1,\infty)$ then $\mathbf{B}_{\Omega}$ is also bounded on $L^p(\Omega)$,
where $\Omega=\{(z,w)\in \mathbb{C}^2~|~ z\in\mathbb{D},~  |w|^2<\omega(z)\}$.
\end{proposition}

\begin{proof}
Recall that $K_m(z,t)$ is the weighted Bergman kernel for the weight $|g(z)|^{2(m+1)}$ so we can apply observation \eqref{kernelrelation} to the kernels $K_m(z,t)$

\begin{align*}
K_m(z,t)&=\frac{1}{g(z)^{m+1}}B_1(z,t)\frac{1}{\overline{g(t)^{m+1}}}\\
&=\frac{1}{g(z)^{m}}B_{\omega}(z,t)\frac{1}{\overline{g(t)^{m}}}.
\end{align*}
Hence, we get
\begin{align*}
B_{\Omega}\left[(z,w),(t,s)\right]&=B_{\omega}(z,t)\sum_{m=0}^{\infty}(2m+2)\left(\frac{w\overline{s}}{g(z)\overline{g(t)}}\right)^m\\
&=B_{\omega}(z,t)B_1\left(\frac{w}{g(z)},\frac{s}{g(t)}\right).\\
\end{align*}

We recognize the sum as the representation of the ordinary Bergman kernel on $\mathbb{D}$ (up to a constant). Therefore, we can express $\BOm$ as a combination of operators involving $\Bom$ and $\mathbf{B}_{1}$.

Indeed, by the integral representation of $\BOm$ and the identity for $B_{\Omega}$ above,

\begin{align*}
\mathbf{B}_{\Omega}F(z,w)&=\int_{\Omega}B_{\Omega}\left[(z,w),(t,s)\right]F(t,s)dV(t,s)\\
&=\int_{\Omega}B_{\omega}(z,t)      B_1\left(\frac{w}{g(z)},\frac{s}{g(t)}\right)     F(t,s)dV(t,s)\\
&=\int_{\mathbb{D}}B_{\omega}(z,t)\int_{|s|^2<|g(t)|^2}B_1\left(\frac{w}{g(z)},\frac{s}{g(t)}\right) F(t,s)dA(s)dA(t)\\
&=\int_{\mathbb{D}}B_{\omega}(z,t)\int_{\mathbb{D}}B_1\left(\frac{w}{g(z)},\sigma\right)     F(t,g(t)\sigma)|g(t)|^2dA(\sigma)dA(t)
\end{align*}
where we make the change of variable $\sigma=\frac{s}{g(t)}$. Next, we compute $||\mathbf{B}_{\Omega}F||^p_{p,\Omega}$ by using this identity and writing the integral on $\Omega$ as an iterated integral and making the change of variable $u=\frac{w}{g(z)}:$
\begin{align*}
&||\mathbf{B}_{\Omega}F||^p_{p,\Omega}=\int_{\Omega}|\mathbf{B}_{\Omega}F(z,w)|^pdV(z,w)\\
&=\int_{\mathbb{D}}\int_{|w|^2<|g(z)|^2}|\mathbf{B}_{\Omega}F(z,w)|^pdA(w)dA(z)\\
&=\int_{\mathbb{D}}\int_{|w|^2<|g(z)|^2}\left|\int_{\mathbb{D}}B_{\omega}(z,t)\int_{\mathbb{D}}B_1(\frac{w}{g(z)},\sigma)F(t,g(t)\sigma)|g(t)|^2dA(\sigma)dA(t)\right|^pdA(w)dA(z)\\
&=\int_{\mathbb{D}}\int_{\mathbb{D}}\left|\int_{\mathbb{D}}B_{\omega}(z,t)\int_{\mathbb{D}}B_1(u,\sigma)F(t,g(t)\sigma)|g(t)|^2dA(\sigma)dA(t)\right|^pdA(u)|g(z)|^2dA(z).\\
\end{align*}

We change the order of integration and integrate with respect to $z$ first. Also, we notice that the expression in braces below is the weighted $p-$norm of a projected function. Furthermore, $\Bom$ is $L^p-$bounded (by the hypothesis), so we get
\begin{align*}
&||\mathbf{B}_{\Omega}F||^p_{p,\Omega}=\\&=\int_{\mathbb{D}}\left\{\int_{\mathbb{D}}\left|\int_{\mathbb{D}}B_{\omega}(z,t)\left[\int_{\mathbb{D}}B_1(u,\sigma)F(t,g(t)\sigma)dA(\sigma)\right]|g(t)|^2dA(t)\right|^p|g(z)|^2dA(z)\right\}dA(u)\\
&\lesssim\int_{\mathbb{D}}\int_{\mathbb{D}}\left|\int_{\mathbb{D}}B_1(u,\sigma)F(t,g(t)\sigma)dA(\sigma)\right|^p|g(t)|^2dA(t)dA(u).\\
\end{align*}

{\allowdisplaybreaks
Once again, we change order of integration and integrate with respect to $u$ first. We notice the same thing above for $\mathbf{B}_1$ now i.e. the expression in braces is the weighted $p-$norm of a projected function and $\mathbf{B}_1$ is $L^p-$bounded, so we get

\begin{align*}
||\mathbf{B}_{\Omega}F||^p_{p,\Omega}&\lesssim\int_{\mathbb{D}}\int_{\mathbb{D}}\left|\int_{\mathbb{D}}B_1(u,\sigma)F(t,g(t)\sigma)dA(\sigma)\right|^p|g(t)|^2dA(t)dA(u)\\
&\lesssim\int_{\mathbb{D}}\int_{\mathbb{D}}\left|F(t,g(t)\sigma)\right|^pdA(\sigma)|g(t)|^2dA(t).\\
&=\int_{\mathbb{D}}\int_{|s|^2<|g(t)|^2}\left|F(t,s)\right|^pdA(s)dA(t)\\
&=\int_{\Omega}|F(t,s)|^pdV(t,s)\\
&=||F||^p_{p,\Omega}.
\end{align*}}
Therefore, we finally get
$$||\mathbf{B}_{\Omega}F||^p_{p,\Omega}\lesssim||F||^p_{p,\Omega}.$$
 
Note that to justify the changes of order of integrations, we can start with a polynomial $F$ and use the fact that polynomials in $(z,\bar{z},w,\bar{w})$ are dense in $L^p(\Omega)$.\\

\end{proof}
\noindent When we combine the last two theorems we get the following corollary.
\begin{corollary}\label{equivalence}
Let $g$ be a holomorphic function on $\mathbb{D}$ which does not vanish inside $\mathbb{D}$ and let $\omega=|g|^2$ be the weight. Then $\Bom$ is bounded on $L^{p}(\omega)$ for some $p\in(1,\infty)$ if and only if $\mathbf{B}_{\Omega}$ is bounded on $L^p(\Omega).$\\
\end{corollary}

This corollary with the weights constructed in Corollary \ref{Type3Corollary} and in Remark \ref{remarkF} establishes the proof of Theorem \ref{Three}. In particular, for any $p_0>2$ if
\begin{equation*}
\Omega_{p_0}=\left\{(z,w)\in \mathbb{C}^2~|~ z\in\mathbb{D},~  |w|^2<|z-1|^{\frac{4}{p_0-2}}\right\}.
\end{equation*}
Then the ordinary Bergman projection of $\Omega_{p_0}$ is bounded on $L^p(\Omega_{p_0})$ if and only if $p\in (q_0,p_0)$.\\

Theorem \ref{firsthalf} combined with the examples of weights in Theorem \ref{Type2mainIntroduction} proves Theorem \ref{Two}. In particular, for any $A\geq 0, B>0, \alpha>0$ if
\begin{equation*}
\Omega_{A,B,\alpha}=\left\{(z,w)\in \mathbb{C}^2~|~ z\in\mathbb{D},~  |w|^2<(1-|z|^2)^A\exp\left(\frac{-B}{(1-|z|^2)^{\alpha}}\right)\right\}.
\end{equation*}
Then the ordinary Bergman projection of $\Omega_{A,B,\alpha}$ is bounded on $L^p(\Omega_{A,B,\alpha})$ if and only if $p=2$.\\

\newpage
\appendix
\section{Details of Remark \ref{remarkF}}

Let $F(\zeta)=\frac{-2i}{(i+\zeta)^2}\left(\frac{-2\zeta}{i+\zeta}\right)^{\frac{2}{p_0-2}}$ for some $p_0>2$. $F$ is a non-vanishing holomorphic function on $\uplane$. For this choice of $F$, we get $\omega=|(F\circ\psi(z))\psi'(z)|^2=|z-1|^{\frac{4}{p_0-2}}$. 

By Corollary \ref{F}, $\Bom$ is bounded on $L^p(\mathbb{D}, \omega)$ if and only if $|F|^{2-p}\in A^+_p(\mathbb{R}^2_+).$ Our goal in this appendix is to show that, indeed
\begin{align*}
|F(\zeta)|^{2-p}&=\left(\frac{2}{|i+\zeta|^2}\left(\frac{|2\zeta|}{|i+\zeta|}\right)^{\frac{2}{p_0-2}}\right)^{2-p}\\ &\sim |\zeta|^{\frac{4-2p}{p_0-2}}|i+\zeta|^{\frac{(2p-4)(p_0-1)}{p_0-2}}\\ 
&\in A^+_p(\mathbb{R}^2_+)
\end{align*}
only for $p\in (q_0, p_0)$.

By Definition \ref{def}, this is equivalent to show that there exists $C=C(p)>0$ such that
\begin{align*}
\frac{1}{|D\cap\mathbb{R}^2_+|^p}&\left(\int_{D\cap\mathbb{R}^2_+}|\zeta|^{\frac{4-2p}{p_0-2}}|i+\zeta|^{\frac{(2p-4)(p_0-1)}{p_0-2}}dA(\zeta)\right)\\
&\left(\int_{D\cap\mathbb{R}^2_+}|\zeta|^{\frac{4-2p}{(p_0-2)(1-p)}}|i+\zeta|^{\frac{(2p-4)(p_0-1)}{(p_0-2)(1-p)}}
dA(\zeta)\right)^{p-1} \leq C
\end{align*}
for any disc $D=D(x_0,R)$, where $x_0\in\mathbb{R}$ and $R>0$, if and only if $p\in (q_0, p_0)$. For convenience, we label the first integral $I_1$ and the second one $I_2$.\\

We start with $p\geq p_0$. In this case, $\frac{4-2p}{p_0-2}=-2-\epsilon$ for some $\epsilon>0$ and therefore $$\int_{D\cap\mathbb{R}^2_+}|\zeta|^{\frac{4-2p}{p_0-2}}|i+\zeta|^{\frac{(2p-4)(p_0-1)}{p_0-2}}dA(\zeta)=\infty$$
for discs $D$ centered at $\zeta=0$. Thus, $|F(\zeta)|^{2-p} \not \in A^+_p(\mathbb{R}^2_+)$ for $p\geq p_0$ and $\Bom$ is unbounded on $L^p(\mathbb{D}, \omega)$ for $p\geq p_0$.\\

Next, we consider $2\leq p <p_0$. When $p=2$ the estimate above holds trivially since $4-2p=0$.

Given any $D=D(x_0,R)$, we split up to the following cases.\\

\noindent \textit{Case 1:} $D\cap D(0,2R)$ is not empty and $R<2$. In this case, $D\subset D(0,4R)$ and

\begin{align*}
I_1&\leq \int_{D(0,4R)}|\zeta|^{\frac{4-2p}{p_0-2}}|i+\zeta|^{\frac{(2p-4)(p_0-1)}{p_0-2}}dA(\zeta)\\
&\leq M_R^{\frac{(2p-4)(p_0-1)}{p_0-2}}\int_{D(0,4R)}|\zeta|^{\frac{4-2p}{p_0-2}}dA(\zeta)\\
&\leq M_R^{\frac{(2p-4)(p_0-1)}{p_0-2}}\frac{R^{2+\frac{4-2p}{p_0-2}}}{2+\frac{4-2p}{p_0-2}}\\
&\lesssim M_R^{\frac{(2p-4)(p_0-1)}{p_0-2}}R^{\frac{2p_0-2p}{p_0-2}}
\end{align*}
where $M_R=\max_{D(0,4R)}|i+\zeta|$.

Also, we have

\begin{align*}
I_2&\leq \int_{D(0,4R)}|\zeta|^{\frac{4-2p}{(p_0-2)(1-p)}}|i+\zeta|^{\frac{(2p-4)(p_0-1)}{(p_0-2)(1-p)}}
dA(\zeta)\\
&\leq m_R^{\frac{(2p-4)(p_0-1)}{(p_0-2)(1-p)}}\int_{D(0,4R)}|\zeta|^{\frac{4-2p}{(p_0-2)(1-p)}}dA(\zeta)\\
&\leq m_R^{\frac{(2p-4)(p_0-1)}{(p_0-2)(1-p)}} \frac{R^{2+\frac{4-2p}{(p_0-2)(1-p)}}}{2+\frac{4-2p}{(p_0-2)(1-p)}}\\
&\lesssim m_R^{\frac{(2p-4)(p_0-1)}{(p_0-2)(1-p)}} R^{2+\frac{4-2p}{(p_0-2)(1-p)}}
\end{align*}
where $m_R=\min_{D(0,4R)}|i+\zeta|$.

Hence, we get

\begin{align*}
I_1(I_2)^{p-1}&\lesssim M_R^{\frac{(2p-4)(p_0-1)}{p_0-2}}R^{\frac{2p_0-2p}{p_0-2}}m_R^{\frac{-(2p-4)(p_0-1)}{p_0-2}} R^{2p-2-\frac{4-2p}{p_0-2}}\\
&=\left(\frac{M_R}{m_R}\right)^{\frac{(2p-4)(p_0-1)}{p_0-2}} R^{2p}
\end{align*}
and finally we get
$$\frac{1}{|D|^p}I_1(I_2)^{p-1}\leq C_p\left(\frac{M_R}{m_R}\right)^{\frac{(2p-4)(p_0-1)}{p_0-2}}.$$

For $R<2$, the quantities $M_R$ and $m_R$ are comparable so we get something finite on the right hand side.\\

\noindent \textit{Case 2:} $D\cap D(0,2R)$ is not empty and $R\geq2$. In this case, $D\subset D(0,4R)$ and $D\subset D(-i,5R)$. We use the H\"older's inequality to get 

\begin{align*}
I_1&\leq \left( \int_{D(0,4R)} |\zeta|^{t\frac{4-2p}{p_0-2}}dA(\zeta)\right)^{\frac{1}{t}} \left( \int_{D(-i,5R)}|i+\zeta|^{\frac{t}{t-1}\frac{(4p-8)(p_0-1)}{p_0-2}}dA(\zeta)\right)^{\frac{t-1}{t}}\\
&\lesssim \left(R^{2+t\frac{4-2p}{p_0-2}}\right)^{\frac{1}{t}}\left( R^{2+\frac{t}{t-1}\frac{(4p-8)(p_0-1)}{p_0-2}}\right)^{\frac{t-1}{t}}\\
&=R^{2+\frac{4-2p}{p_0-2}+\frac{(4p-8)(p_0-1)}{p_0-2}}.
\end{align*}
for some $t>1$ such that the first integral above is finite.

On the other hand, again by H\"older's inequality we get

\begin{align*}
I_2&\leq \left( \int_{D(0,4R)}|\zeta|^{t\frac{4-2p}{(p_0-2)(1-p)}}dA(\zeta)\right)^{\frac{1}{t}}\left(\int_{D(-i,
 5R)}|i+\zeta|^{\frac{t}{t-1}\frac{(2p-4)(p_0-1)}{(p_0-2)(1-p)}}
dA(\zeta)\right)^{\frac{t-1}{t}}\\
&\lesssim \left( R^{2+t\frac{4-2p}{(p_0-2)(1-p)}}\right)^{\frac{1}{t}}\left(R^{2+\frac{t}{t-1}\frac{(2p-4)(p_0-1)}{(p_0-2)(1-p)}}\right)^{\frac{t-1}{t}}\\
&=R^{2+\frac{4-2p}{(p_0-2)(1-p)}+\frac{(2p-4)(p_0-1)}{(p_0-2)(1-p)}}.
\end{align*}
for some $t>1$ such that the first integral above is finite.

Combining these two estimates, we obtain $I_1(I_2)^{p-1} \lesssim R^{2p}$ and  $$\frac{1}{|D|^p}I_1(I_2)^{p-1}\leq C_p.$$\\

\noindent \textit{Case 3:} $D\cap D(0,2R)$ is empty. For this case, the crucial observation is the following. If 
\begin{align*}
N_R&= \max_{D}|\zeta| ~\text{ and }n_R=\min_{D}|\zeta|\\
K_R&= \max_{D}|i+\zeta| \text{ and }k_R=\min_{D}|i+\zeta|
\end{align*}
then for any $R>0$, the quantities $N_R$ and $n_R$ and the quantities $K_R$ and $k_R$ are comparable to each other. 
Therefore, we get
\begin{align*}
I_1\lesssim |D|n_R^{\frac{4-2p}{p_0-2}}K_R^{\frac{(4p-8)(p_0-1)}{p_0-2}}~\text{ and }~
I_2\lesssim |D|N_R^{\frac{4-2p}{(p_0-2)(1-p)}}k_R^{\frac{(2p-4)(p_0-1)}{(p_0-2)(1-p)}}.
\end{align*}
These give us
\begin{align*} 
I_1(I_2)^{p-1} \lesssim |D|^p\left(\frac{n_R}{N_R}\right)^{\frac{4-2p}{p_0-2}}\left(\frac{K_R}{k_R}\right)^{\frac{(4p-8)(p_0-1)}{p_0-2}}
~\text{ and }~
\frac{1}{|D|^p}I_1(I_2)^{p-1} &\leq C_p.\\
\end{align*}

Therefore, in all three cases $\frac{1}{|D|^p}I_1(I_2)^{p-1}$ is bounded and $|F(\zeta)|^{2-p} \in A^+_p(\mathbb{R}^2_+)$ for $2\leq p < p_0$ and $\Bom$ is bounded on $L^p(\mathbb{D}, \omega)$ for $2\leq p < p_0$.\\

Duality and the self adjointness of $\Bom$ concludes that for this choice of $F$ and $\omega$, the weighted Bergman projection $\Bom$ is bounded on $L^p(\mathbb{D}, \omega)$ if and only if $p \in (q_0,p_0)$ where $\frac{1}{p_0}+\frac{1}{q_0}=1$.\\

\vskip 1cm
\bibliographystyle{alpha}
\bibliography{ThesisBib}
\vskip 1cm
\end{document}